\documentclass[12pt]{amsart}

\usepackage{amssymb}

\usepackage{enumitem}

\usepackage{graphicx}
\usepackage[english]{babel}
\usepackage{indentfirst} %fa iniziare i paragrafi allineati
\usepackage{latexsym}
\usepackage{amsmath}
\usepackage{amsthm}
\usepackage[hidelinks]{hyperref}
\usepackage[dvipsnames]{xcolor}
\makeatletter
\@namedef{subjclassname@2020}{%
  \textup{2020} Mathematics Subject Classification}
\makeatother

\newtheorem{theorem}{Theorem}[section]
\newtheorem{corollary}[theorem]{Corollary}
\newtheorem{lemma}[theorem]{Lemma}

\newtheorem{mainthm}[theorem]{Main Theorem}

\theoremstyle{definition}
\newtheorem{definition}[theorem]{Definition}

\newtheorem{example}[theorem]{Example}

\numberwithin{equation}{section}
\usepackage{csquotes}

\usepackage{cite}

\newcommand{\U}{\mathcal{U}}
\newcommand{\N}{\mathbb{N}}
\newcommand{\Com}{\mathbb{C}}
\newcommand{\Ze}{\mathbb{Z}}
\newcommand{\Real}{\mathbb{R}}
\newcommand{\Rat}{\mathbb{Q}}

\newcommand{\thistheoremname}{}
\newtheorem*{genericthm*}{\thistheoremname}
\newenvironment{namedthm*}[1]
  {\renewcommand{\thistheoremname}{#1}%
   \begin{genericthm*}}
  {\end{genericthm*}}

\begin{document}

%%%%% To ease editing, for IMPAN journals add:

\baselineskip=17pt

%%%%%%%%%%%%%%%%

\title[A limiting result for the Ramsey theory of functional equations]{A limiting result for the Ramsey theory of functional equations}

\author[P. H. Arruda]{Paulo Henrique Arruda }
\address{Fakult\"{a}t f\"{u}r Mathematik\\ Universit\"{a}t Wien\\
Oskar-Morgenstern-Platz 1 \\1090 Vienna, Austria.}
\email{paulo.arruda@univie.ac.at}

\author[L. Luperi Baglini]{Lorenzo Luperi Baglini}
\address{Dipartimento di Matematica\\ Universit\`{a} di Milano\\ Via Saldini 50 \\20133 Milano, Italy.}
\email{lorenzo.luperi@unimi.it}


\date{}

\begin{abstract}
We study systems of functional equations whose solutions can be parameterized in function of one of the variables; our main result proves that the partition regularity (PR) of such systems can be completely characterized by the existence of constant solutions. As applications of this result, we prove the following:
\begin{itemize}
 \item A complete characterization of PR systems of Diophantine equations in two variables over $\N$. In particular, we prove that the only infinitely PR irreducible equation in two variables is $x=y$;
    \item PR of $S$-unit equations and the failure of Rado's Theorem for finitely generated multiplicative subgroups of $\mathbb{C}$; and 
    \item a complete characterization of the PR of two classes of polynomial exponential equations.
\end{itemize}
\end{abstract}

\maketitle

\section{Introduction}

This paper deals with the study of the so-called \emph{partition regularity of systems of equations}. 

\begin{definition}\label{definition:partition_regularity1}
Let $S$ be an infinite set, let $n,m\in\N$ and let $f_1,\dots,f_m:S^{n}\to S$. Let
    \begin{equation}\label{syst}
        \sigma(x_1,\dots,x_n) = 
        \left\{
        \begin{matrix}
            f_1(x_1,\dots,x_{n}), \\
            \vdots\\
            f_m(x_1,\dots,x_{n}).
        \end{matrix}
        \right.
    \end{equation}
    Given $\boldsymbol{s}=(s_1,\dots,s_m)\in S^m$, we say that the system $\sigma\left(x_1,\dots,x_{n}\right)=\boldsymbol{s}$ is partition regular (abbreviated as PR from now on) on $A\subseteq S$ if for all $r\in\N$ and all $c: S\to\{1,\dots,r\}$, there are $i\in\{1,\dots,r\}$ and $a_1,\dots,a_n\in c_{i}^{-1}(A)$ such that $\sigma(a_1,\dots,a_n)=\boldsymbol{s}$. 
\end{definition}

Whenever the context is clear, we write $\sigma\left(x_{1},\dots,x_{n}\right) \allowbreak =\boldsymbol{s}$ to mean that it is as in (\ref{syst}). When the system consists just of one equation, we will say that the equation is PR; if the equation has the form $P\left(x_{1},\dots,x_{n}\right)=0$ where $P$ is a polynomial, we say also that $P$ is PR to mean that the equation $P\left(x_{1},\dots,x_{n}\right)=0$ is PR. Functions $c:A\to\{1,\dots,r\}$ are usually referred to as colorings, and the condition $c(a_1)=\dots =c(a_{n})$ is usually rephrased as $a_{1},\dots,a_{n}$ are $c$-monochromatic. Given a $k\in\N$, $[k]$ denotes the set $\{1,\dots,k\}$.

It is well known that the study of PR properties is intertwined with the study of the ultrafilters over a set $S$. The monograph \cite{HindmanStrauss2011} provides an introduction to the topic and we assume that the reader has familiarity with the basic theory of ultrafilters. This relationship is made explicit by the following result, see \cite[Proposition 1.8]{DiNassoLuperiBaglini2018} or, more generally, \cite[Theorem 3.11]{HindmanStrauss2011}:

\begin{theorem}\label{theorem:pr_via_ultrafilters} Under the hypothesis of Definition \ref{definition:partition_regularity1}, the system \eqref{syst} is partition regular over $S$ if and only if there is an ultrafilter $\U$ on $\beta S$ such that for all $A\in\U$ there are $a_1,\dots,a_{n}\in A$ satisfying $\sigma\left(a_1,\dots,a_{n}\right)=\boldsymbol{s}$.
\end{theorem}

An ultrafilter $\U$ with the property of Theorem \ref{theorem:pr_via_ultrafilters} will be called a \emph{witness} of the partition regularity of $\sigma\left(x_{1},\dots,x_{n}\right)=0$.

The general problem of the partition regularity of systems of Diophantine equations, especially\footnote{To avoid trivialities, in this paper we convene that $\N=\{1,2,3,\dots\}$.} on $\N$, has long been studied. The linear case was settled by R. Rado in 1933 in terms of the so-called columns condition:

\begin{definition}
Let $R$ be a commutative ring, $A$ be a $m\times n$ matrix with entries in $R$ and $C_1,\dots,C_m$ be the columns of $A$; we say that $A$ satisfies the columns condition if there is a partition $I_0,\dots,I_r$ of $[m]$ such that
\begin{enumerate}
    \item $\sum_{i\in I_0}C_i=\boldsymbol{0}$; and 
    \item given any $u\in [r]$, $\sum_{i\in I_u}C_i \in \operatorname{span}_K\{C_j\mid j\in I_0\cup\dots\cup I_{u-1}\}$, where $K$ is the field of fractions of $R$ and $\operatorname{span}_K\{ C_j\mid j\in I\}$ denotes the vector subspace generated by $\{C_j\mid j\in I\}$.
\end{enumerate}
\end{definition}
Rado's Theorem for linear homogeneous systems, proven by Rado in \cite{Rado1933}, reads as follows:
\begin{namedthm*}{Homogeneous Rado's Theorem}\label{theorem:RadosTheorem}
Given an $m\times n$ matrix $A$ with rational entries, the system $A\left(x_{1},\dots,x_{n}\right)=\boldsymbol{0}$ is partition regular over $\N$ if and only if $A$ satisfies the columns condition.
\end{namedthm*}

In particular, for a single equation, Rado's Theorem tells that the equation $c_{1}x_{1}+\dots+c_{n}x_{n}=0$ is partition regular over $\N$ if and only if the following condition (known as Rado's Condition) holds: there exists a nonempty set $J\subseteq\{1,\dots,n\}$ such that $\sum_{j\in J} c_{j}=0$.

Whilst there are plenty of results regarding various aspects of partition regularity of finite and infinite systems of linear equations in the literature, progress on the general case of functional equations has been scarce and mostly (but not exclusively) concentrated in the past few years. For nonlinear polynomial equations we refer to the introductory section of \cite{DiNassoLuperiBaglini2018} for a complete list of nonlinear results proven until 2018, and to \cite{DiNassoRiggio2018,BarretMoreiraLupiniMoreira2021, FarhangiMagner2021, Moreira2017, Prendiville2021} for the latest advancements we are aware of.

Rado's Theorem for non-homogeneous linear equations, again proven in \cite{Rado1933}, reads as follows:
\begin{namedthm*}{Non-homogeneous Rado's Theorem}\label{theorem:rado_nonhom_linear}
Let $A$ be an $m\times n$ matrix with rational entries and $\boldsymbol{b}\in\Ze^m$. Then the system $A(x_1,\dots,x_n)=\boldsymbol{b}$ is PR over $\N$ if and only if either 
\begin{enumerate}
    \item There is an $a\in \N$ such that $A(a,\dots,a)=\boldsymbol{b}$; or 
    \item The matrix $A$ satisfies the columns condition and there exists $a\in \mathbb{Z}$ such that $A(a,\dots,a)=\boldsymbol{0}$.
\end{enumerate}
\end{namedthm*}

\begin{definition}
    Under the notations of Definition \eqref{definition:partition_regularity1}, we say that the system $\sigma(x_1,\dots,x_n)=\boldsymbol{s}$ admits a constant solution if one can find an $a\in A$ such that $\sigma(a,\dots,a)=\boldsymbol{s}$.
\end{definition}

In particular, for a single equation $c_1x_1+\cdots +c_nx_n=b$, Rado's Theorem states that this equation is PR over $\N$ iff either it has a constant solution in $\N$, or it has a constant solution in $\Ze$ and there is a non-empty $J\subseteq [n]$ such that $\sum_{i\in J}c_i=0$. 

It is easy to see that if a system $\sigma\left(x_1,\dots,x_{n}\right)=\boldsymbol{s}$ admits a constant solution, then it is PR, as constant solutions are obviously monochromatic\footnote{In fact, in many papers only strengthened notions of partition regularity are considered where only nonconstant solutions or even more restrictive limitations are imposed. See e.g. \cite{DiNassoLuperiBaglini2018}.}. In \cite{HindmanLeader2006}, N. Hindman and I. Leader proved conditions under which linear systems are non-trivially PR, i.e. PR linear systems that admit solutions other than constants. In particular, when $n\geq 3$ they proved that a linear equation $c_{1}x_{1}+\dots+c_{n}x_{n}=0$ is PR if and only if it is non-trivially PR\footnote{Actually, something strong holds: for $n\geq 3$, one has that in any finite coloring of $\N$ there exist monochromatic injective solutions of the equation, namely solutions where all the entries are mutually distinct. When $n=2$ of course this cannot happen as, by Rado's Theorem, the only PR linear equation in two variables is $x-y=0$. We will show in Section \ref{section:applications} that the same remains true if we substitute "linear" with "irreducible".}. By contrast, our main result shows that whenever the solutions of a system $\sigma(x_1,\dots,x_{n})=\boldsymbol{s}$ can be parameterized in function of one of the variables\footnote{Henceforth assumed without loss of generality to be $(n+1)$-th variable.}, to be PR is equivalent to admit a constant solution. 

\begin{mainthm}\label{theorem:main_theorem}
Let $S$ be an infinite set, $m,n\in\N$, $\boldsymbol{s}=(s_1,\dots,s_m)\in S^m$, and $f_1,\dots,f_m:S^{n+1}\to S$. Let $\sigma(x_1,\dots,x_n,x_{n+1})=\boldsymbol{s}$ be a system as in \eqref{syst}. Suppose that there exists $k\in\N$ such that for all $s\in S$ the number of solutions in the variables $x_1,\dots,x_n$ of the system $\sigma\left(x_{1},\dots,x_{n},s\right)=\boldsymbol{s}$ is at most $k$. Then, the following are equivalent:
\begin{enumerate}
    \item The system is PR over $S$;
    \item The set of witnesses of the PR of the system is non-empty and for all witnesses $\U$ of the PR of the system there exists a $B\in \U$ such that the only solutions of the system in $B$ are constant; and
    \item There is a coloring $c_0$ of $S$ such that: (a) the system has a $c_0$-monochromatic solution, and (b) all $c_0$-monochromatic solutions the system are constant.
\end{enumerate}
 
\end{mainthm}

Although we will apply our result to some explicit classes of nonlinear or exponential polynomials, we notice that Theorem \ref{theorem:main_theorem} is purely set-theoretical: it holds for arbitrary systems of functional equations on infinite sets. Theorem \ref{theorem:main_theorem} is proven in Section \ref{section:proofs} and its applications are listed in Section \ref{section:applications}, which we briefly discuss below.

The first application, see Theorem \ref{theorem:main2}, is an extension of Rado's Theorem for the PR of systems of Diophantine equations in two variables over $\N$, thus closing the problem for two variables. Summarizing, to the best of our knowledge, the only nonlinear classes whose partition regularity over $\N$ has been completely characterized are\footnote{The situation is slightly different when it comes to the partition regularity over $\Real,\Com$; see e.g. \cite{LuperiBaglini2021} for a discussion of this fact.}${}^{,}$\footnote{The partition regularity of the cases discussed below disconsider constant solutions, i.e. these characterizations ensure the existence of non-constant solutions.}:
	\begin{enumerate} 
		\item equations of the form $\sum_{i=1}^{n}c_{i}x_{i}=P(y)$, for $P(y)$ nonlinear polynomial with no constant term (see \cite[Corollary 3.14]{DiNassoLuperiBaglini2018}); these equations are partition regular if and only if the linear part satisfies Rado's condition, namely if and only if there exists a nonempty $I\subseteq [n]$ such that $\sum_{i\in I}c_{i}=0$;
		\item equations of the form $\sum_{i=1}^{h(n)}c_{i}x_{i}^{n}=0$ for $h(n)$ large enough (namely, $h(n)\geq (1+o(n))n\log n$, see \cite{ChowLindqvistPrendiville2021} for details), which are partition regular over $\N$ if and only if there exists a nonempty $I\subseteq [n]$ such that $\sum_{i\in I}c_{i}=0$;
		\item equations of the form $\sum_{i=1}^{s}a_{i}x{i}^{2}=\sum_{j=1}^{t}b_{j}y_{j}$, where the $a_{i},b_{j}$ are non-zero integers, except those of the form $a_{1}\left(x_{1}^{2}-x_{2}^{2}\right)=a_{2}x_{3}^{2}+b_{1}y_{1}$, for which the partition regularity is still unknown (see \cite[Theorem 1.10]{Prendiville2021}). These equations are partition regular if and only if there exists $I\neq\emptyset$ such that $\sum_{i\in I}a_{i}=0$ or $\sum_{j\in I}b_{j}=0$.
	\end{enumerate}

Rado's Theorem is known to hold in different settings such as subrings of $\Com$ (see \cite{Rado1945}), several families of commutative rings (see \cite{BergelsonDeuberHindmanLefmann1994}) and for Abelian groups (see \cite{Deuber1975}); recently, it has also been proven for infinite integral domains (see \cite{ByszewskiKrawczyk2021}). Our second application, see Corollary \ref{corollary:S-unit-equations_Non_PR}, concerns the partition regularity of linear equations over finite rank multiplicative subgroups of $\Com^\times$. We observe an interesting phenomenon in this setting: the failure of Rado's Theorem for this class of equations; in particular, we prove that such sets are additive combinatorial small in the sense that they do not contain solutions of Schur's Equation or $3$-terms arithmetic progressions.

Our third application deals with polyexponential equations over algebraic number fields\footnote{I.e. a finite field extension of the field of rational numbers.}, see Theorem \ref{theorem:PR_exponential_equation}. In recent years, starting with the work of A. Sisto \cite{Sisto2011} and then extended by \cite{Sahasrabudhe2018a}, a relevant problem in the area has been the characterization of the partition regularity of polyexponential equations (and configurations), namely equations obtained by finite compositions of exponentiation, sum and product. This problem is far from being well understood, with only a few positive and even less negative results proven, such as 
\begin{enumerate}
    \item The configuration $\{x,y,x^y,xy\}$ is PR (see \cite[Theorem 2]{Sahasrabudhe2018a}). In particular, from this it follows that the equations $x=y^z$ and $xy=z^w$ are PR; and 
    \item The configuration $\{x,y,x^y,a,b,a+b\}$ is not PR (see \cite[Theorem 6]{Sahasrabudhe2018a}). 
\end{enumerate}

Applying Theorem \ref{theorem:main_theorem}, we study two classes of polyexponential equations whose characters are mutually coprime and characterize their PR in terms of the existence of constant solutions; precisely, we consider equations of the forms
\begin{equation*}
    P_1(\boldsymbol{x})f_1(y)\boldsymbol{\alpha}_1^{\boldsymbol{x}}+\cdots +P_m(\boldsymbol{x})f_m(y)\boldsymbol{\alpha}_m^{\boldsymbol{x}}=0,
\end{equation*}
and 
\begin{equation*}
 Q_1(\boldsymbol{x},y)\boldsymbol{\alpha}_1^{\boldsymbol{x}}+\cdots +Q_m(\boldsymbol{x},y)\boldsymbol{\alpha}_m^{\boldsymbol{x}}=0,
\end{equation*}
where all $P_i$'s and $Q_i$'s are polynomial over the rational numbers, $\boldsymbol{\alpha}_i=(\alpha_{i1},\dots,\alpha_{in})\in(\Ze\setminus\{0\})^n$ are such that 
\begin{equation*}
    \gcd\{\alpha_{ij}\mid i\in[m] \text{ and } j\in[m]\}=1,
\end{equation*}
and $\boldsymbol{\alpha}_i^{\boldsymbol{x}}=\alpha_{i1}^{x_1}\cdots\alpha_{in}^{x_n}$.
\section{Proof of the main result}\label{section:proofs}

In this section, we provide the proof of our main result by means of ultrafilters. It will be necessary to quantify how much \emph{injectivity} we can possibly find in a given solution of a given system; more precisely:

\begin{definition}\label{definition:injectivity}\cite[Definition 1.4]{DiNassoLuperiBaglini2018}
Under the conditions of Definition \ref{definition:partition_regularity1}, we say that the system $\sigma(x_1,\dots,x_n)=\boldsymbol{s}$ is PR over $S$ with injectivity $\geq r$ ($r\in[n]$), if for all colorings $c$ of $S$ one can find $c$-monochromatic $a_1,\dots,a_n\in S$ such that $|\{a_1,\dots,a_n\}|\geq r$.
\end{definition}

As expected, being PR with a certain injectivity is equivalent to having a witness of the PR with such injectivity:

\begin{theorem}\cite[Proposition 1.8]{DiNassoLuperiBaglini2018}\label{theorem:lorenzo&mauro}
Under the conditions of Definition \ref{definition:partition_regularity1}, the system $\sigma(x_1,\dots,x_n)=\boldsymbol{s}$ is PR over $S$ with injectivity $\geq r$ if and only if there is an ultrafilter $\U\in\beta S$ such that for all $A\in\U$ one can find a solution $a_1,\dots,a_n\in A$ of the system and $|\{a_1,\dots,a_n\}|\geq r$.
\end{theorem}

Additionally, we employ a well-known property of ultrafilters related to fixed points of functions in $\beta S$:

\begin{theorem}\label{fix point}\cite[Theorem 3.35]{HindmanStrauss2011} 
Let $\U\in\beta S$ and let $f:S\rightarrow S$. Let $\overline{f}:\beta S\rightarrow \beta S$ be the continuous extension of $f$ to $\beta S$. Then $\overline{f}(\U)=\U$ if and only if there exists $A\in\U$ such that $f(a)=a$ for all $a\in A$.
\end{theorem}

Before starting the proof of Theorem \ref{theorem:main_theorem} we prefer to isolate the following general Lemma we will need, which basically states that a conjunction of properties is PR if only if at least one of the properties is PR.

\begin{lemma}\label{easy} Let $\U\in\beta S$ and let $\varphi_{1}\left(x_{1},\dots,x_{n}\right),\dots,\varphi_{k}\left(x_{1},\dots,x_{n}\right)$ be properties on $S$. The following are equivalent:

\begin{enumerate}
    \item Given any $A\in\U$ there exists $a_1,\dots,a_n\in A$ and $j\in [k]$ such that $\varphi_j(a_1,\dots,a_n)$ is satisfied; and 
    \item there exists a $j\in[k]$ such that for all $A\in\U$ one can find $a_1,\dots,a_n\in A$ satisfying $\varphi_j(a_1,\dots,a_n)$.
\end{enumerate}
\end{lemma}
\begin{proof}
    $(1)\Rightarrow (2)$ Assuming the negation of $(2)$, for all $j\in[k]$ one can find $A_j\in\U$ satisfying $\neg \varphi_{j}(a_1,\dots,a_n)$ whenever $a_1,\dots,a_n\in A_j$; since $A=A_1\cap \dots\cap A_k\in\U$ and $\neg\varphi_j(a_1,\dots,a_n)$ is satisfied for all $j\in [k]$, we arrive to a contradiction with $(1)$. 

    The implication $(2)\Rightarrow (1)$ is immediate.
\end{proof}

With Theorems \ref{theorem:lorenzo&mauro} and \ref{fix point}, and Lemma \ref{easy}, we can prove Theorem \ref{theorem:main_theorem}.

\begin{proof}[Proof of Theorem \ref{theorem:main_theorem}] (1)$\Rightarrow$(2). By the hypothesis, the following property holds:
\begin{quote}
\begin{description}
    \item[(P$_0$)] There are functions $\psi_1,\dots,\psi_k:S\to S^n$ such that whenever $a_1,\dots,a_n,a_{n+1}\in S$ is a solution of the system $\sigma(x_1,\dots,x_n,x_{n+1})=\boldsymbol{s}$ one can find a $j\in[k]$ satisfying $\psi_j(a_{n+1})=(a_1,\dots,a_n)$.
\end{description}  
\end{quote}
Let thus $\varphi_j(x_1,\dots,x_n,x_{n+1})$ be the formula 
\begin{equation*}
    \sigma(x_1,\dots,x_n,x_{n+1})=\boldsymbol{s} \text{ and } \psi_j(x_{n+1})=(x_1,\dots,x_n).
\end{equation*}
As $\sigma(x_1,\dots,x_n,)=\boldsymbol{s}$ is PR over $S$, it has a witness $\U\in\beta S$ of its PR. Fix this witness. Given $A\in\U$, we claim that there are $j\in[k]$ and $\boldsymbol{a}=(a_1,\dots,a_n,a_{n+1})\in A^{n+1}$ such that $\varphi_j(\boldsymbol{a})$ is satisfied; to prove this claim, we proceed by contradiction: if the contrary happens, there is an $A\in\U$ such that for all $j\in[k]$ and $\boldsymbol{a}\in A^{n+1}$ the property $\varphi_j(\boldsymbol{a})$ is not satisfied. By property (P$_0$) above, this means that $A$ cannot contain a solution of the system $\sigma(x_1,\dots,x_n,x_{n+1})=\boldsymbol{s}$ which contradicts the fact that $\U$ witnesses the PR of such system. 

By Lemma \ref{easy}, there is a $j_0\in[k]$ such that for all $A\in\U$ the following property is satisfied:
\begin{quote}
\begin{description}
    \item[(P$_1$)] There exists an $\boldsymbol{a}\in A^{n+1}$ satisfying $\varphi_{j_0}(\boldsymbol{a})$.
\end{description}  
\end{quote}
Given any $A\in\U$, we claim that the set
\begin{equation*}
    A_{j_0} := \left\{a\in A: \exists (a_1,\dots,a_n)\in A^n \big(\varphi_{j_0}(a_1,\dots,a_n,a)\big)\right\}
\end{equation*}
is a member of $\U$. To prove this claim, we proceed by contradiction: if not, then $A' = A\cap (S\setminus A_{j_0})\in \U$. But it is easily noticeable that $A'$ does not satisfy property (P$_1$), which is a contradiction. 

For any given $A\in \U$ and $i\in[k]$, define 
\begin{equation*}
    A_{i,j_0} = \left\{a\in A : \pi_i\circ \psi_{j_0}(a)\in A\right\},
\end{equation*}
where $\pi_i:S^n\to S$ is the projection onto the $i$th coordinate. One can easily verify that $A_{j_0}\subseteq A_{i,j_0}$ and thus $A_{i,j_0}\in\U$. This proves that for all $i\in[n]$ one has that $\pi_i\circ\psi_{j_0}(\U)=\U$; by Theorem \ref{fix point}, there is a $B_i\in\U$ such that $\pi_i\circ\psi_{j_0}$ is the identity when restricted to $B_i$, and thus the set
\begin{equation*}
    B := \bigcap_{i=1}^{n} (B_i)_{j_0}
\end{equation*}
is the searched set. Indeed, if $(b_1,\dots,b_n,b_{n+1})\in B^{n+1}$ is a solution of the system, by property (P$_1$), one has that $\psi_{j_0}(b_{n+1})=(b_1,\dots,b_n)$. It thus must be the case that for any $i\in[n]$ one has that $b_i = \pi_i\circ\psi_{j_0}(b_{n+1})=b_{n+1}$.

(2)$\Rightarrow$ (3). To prove this implication, we proceed by contradiction. Since $\U$ is a witness of the PR of the system, we know that there are monochromatic solutions of any coloring $c$ of $S$; suppose that for all colorings $c$ of $S$ there is a $c$-monochromatic non-constant solution $a_1,\dots,a_n,a_{n+1}\in S$ to the system. This is easily seen as equivalent to saying that the system $\sigma(x_1,\dots,x_n,x_{n+1})=\boldsymbol{s}$ is PR over $S$ with injectivity $\geq 2$. By Theorem \ref{theorem:lorenzo&mauro}, one can find a witness $\U\in\beta S$ to this fact; but this is a contradiction with the hypothesis as $\U$ must contain a set $B$ in which all solutions of the system are constant.

(3)$\Rightarrow$(1). Since all $c_0$ monochromatic solutions of the system are constant (which exists by the hypothesis), the system is PR for trivial reasons.
\end{proof}

As a trivial consequence, we also get the following result.

\begin{corollary} In the same hypotheses of Theorem \ref{theorem:main_theorem}, for all $\U\in\beta S$ the following are equivalent:
\begin{enumerate}
\item $\U$ is a witness of the partition regularity of $\sigma\left(x_{1},\dots,x_{n}\right)=\boldsymbol{s}$;
\item $\Delta_{\sigma}\in\U$, where
\begin{equation*}
    \Delta_{\sigma} = \{a\in S\mid \sigma(a,\dots,a)=\boldsymbol{s}\}.
\end{equation*}
\end{enumerate}
\end{corollary}

\section{Applications}\label{section:applications}

\subsection{PR of equations in two variables}
We prove that a given system of polynomial equations in two variables over $\Ze$ is PR over $\N$ if and only if it admits constant solutions. In what follows, $\sigma(x,y)=\boldsymbol{0}$ will denote a system of polynomial equations. In such a situation, the best we can hope for in terms of PR is to have a stronger property, that is to any given finite coloring of $\N$ there are infinitely many monochromatic constant solutions; in this case, we say that $\sigma(x,y)=\boldsymbol{0}$ is infinitely PR. As we are going to show, the only irreducible infinitely PR polynomial in $\Ze[x,y]$ is $x-y$. Since the only ingredient needed to show this fact together with Theorem \ref{theorem:main_theorem} is that polynomials in one variable over $\Ze$ have a uniform bound on their number of solutions that only depends on the degree of the polynomial, the same result holds \emph{mutatis mutandis} for any infinite subset of a given infinite integral domain. 

Let us first observe that, when the system consists of homogeneous equations, this result holds trivially as homogeneous polynomials factorize as products of linear polynomials over $\Com$ (or the algebraic closure of the field of fractions of the integral domain), so we could directly conclude by Rado's Theorem for integral domains and the fact that a polynomial is PR if and only if at least one of its irreducible factors is (see e.g. \cite[Theorem 3.7]{LuperiBaglini2015} and \cite[Theorem A]{ByszewskiKrawczyk2021}).

However, Theorem \ref{theorem:main_theorem} allows a very simple proof of the general case.

\begin{theorem}\label{theorem:main2} 
Let $P_{1},\dots,P_{m}\in \Ze\left[x,y\right]$ be polynomials having degree $\geq 1$ and $\sigma(x,y)=\boldsymbol{0}$ the system
    \begin{equation*}\label{equation:system1}
    \left\{
        \begin{matrix}  
            P_{1}\left(x,y\right)=0,\\
            \vdots\\ 
            P_{m}\left(x,y\right)=0.
        \end{matrix}
        \right.
        \tag{$\star$}
    \end{equation*}
The following facts are equivalent:
\begin{enumerate}
\item The system $\sigma(x,y)=\boldsymbol{0}$ has a constant solution;
\item The system $\sigma(x,y)=\boldsymbol{0}$ is PR over $\N$.
\end{enumerate}

\end{theorem} 

\begin{proof} 
Our method allows us to prove a stronger result, namely: let $R$ be any infinite integral domain, let $S\subseteq R$ an infinite set, and let $P_1,\dots,P_m$ be polynomials on two variables over $R$; considering the system \eqref{equation:system1}, we prove that the following are equivalent:

\begin{enumerate}
\item The system $\sigma(x,y)=\boldsymbol{0}$ has a constant solution;
\item The system $\sigma(x,y)=\boldsymbol{0}$ is PR on $S$.
\end{enumerate}

$(1)\Rightarrow (2)$ This is trivial. 

$(2)\Rightarrow (1)$ For each $i\in[m]$ and $s\in S$, consider the monovariate polynomial $\tilde{P}_{i,s}(x) = P_i(x,s)$. Let $k=\max\{\operatorname{deg}\left(P_{i,s}\right): i\in [m]\}$. Then each $P_{i,s}(x)$ has at most $k$ roots over $R$. We can hence conclude the thesis by a direct application of Theorem \ref{theorem:main_theorem}.
\end{proof}

The previous result shows that being PR is equivalent to having at least one constant solution. As an immediate Corollary of Theorem \ref{theorem:main2}, we have the following characterization of being infinitely PR.

\begin{corollary} Under the same notations and hypotheses of Theorem \ref{theorem:main2}, the following are equivalent:
\begin{enumerate}
\item The system $\sigma(x,y)=\boldsymbol{s}$ is infinitely PR over $\N$;
\item $(x-y)$ divides $P_{1}\left(x,y\right),\dots,P_{m}\left(x,y\right)$.
\end{enumerate}
In particular, $x-y$ is the only irreducible infinitely PR polynomial in two variables. 
\end{corollary}

\begin{proof} As before, we can actually prove our theorem in the more general setting where $R$ is any infinite integral domain, $S\subseteq R$ is an infinite set, and $P_1,\dots,P_m$ are polynomials over $R$.

$(1)\Rightarrow (2)$ Let us first observe that a general system of functional equations $\sigma(x_1,\dots,x_n)=\boldsymbol{s}$ is infinitely PR over $S$ iff there is a free ultrafilter witness\footnote{A \emph{principal} ultrafilter on $S$ is an ultrafilter of the form $\{A\subseteq S\mid s\in A\}$; a \emph{free} ultrafilter on $S$ is an ultrafilter that is non-principal. In the case where $S$ is infinite, if $\U$ is free, then all elements of $\U$ are infinite.}. By Theorem \ref{theorem:main2} and the fact that $\U$ does not contain finite sets, the system $\sigma(x,y)=\boldsymbol{s}$ has an infinite amount of constant solutions in any $A\in\U$. Then we can apply B\'ezout Theorem: if $F$ is the algebraic closure of the fields of fractions of $R$, then it must be the case that $\{(x,y)\in S^2\mid x=y\}$ intersect the affine curve determined by the system $\sigma(x,y)=\boldsymbol{s}$ in an infinite amount of points, which consequently implies that $x-y$ divides each $P_i$.

$(2)\Rightarrow (1)$ This is trivial, as it entails that for all $r\in R \ (r,r)$ is a monochromatic solution of $\sigma(x,y)=0$. \end{proof}

Notice that there are plenty of examples of PR polynomials in two variables that are not infinitely PR: for example, for all $n\in\N$ the equation $2x-y=n$ is PR, as it has the constant solution $x=y=n$, but it is not infinitely PR as it does not have an infinite amount of solutions.

\subsection{PR of S-unit equations}
Let $K$ be a field of characteristic $0$. We denote by $K^{\times}$ the multiplicative group $(K\setminus\{0\}, \cdot)$. We say that an Abelian (multiplicatively written) group $\Gamma$ is \emph{finitely generated} if there are $r\in\N$ and $e_1,\dots,e_r\in\Gamma$ such that for all $g\in\Gamma$ one can find $m,n_1,\dots,n_k\in\Ze$ satisfying $x^m=e_1^{n_1}\dots e_k^{n_k}$. The smallest $r\in\N$ such that $\Gamma$ is generated by $r$ elements is called the \emph{rank} of $\Gamma$. We fix $\Gamma$ to be a multiplicative subgroup of $\Com^\times$ of rank $r$. An $S$-unit equation over $\Gamma$ is an equation of the form 
\begin{equation*}
    c_1x_1 + \cdots +c_n x_n = 1
\end{equation*}
where $c_1,\dots,c_n\in \Com^\times$ and unknowns $x_1,\dots,x_n$ are to be searched inside $\Gamma$. The theory of $S$-unit equations is an active research area in Diophantine number theory and its properties have been investigated since the beginning of the XX century. We refer to the book \cite{Evertse2015} as a general reference to the subject. We show below, as an application of Theorem \ref{theorem:main_theorem}, that Rado's Theorem fails to produce monochromatic solutions of linear equations in $3$ variables over $\Gamma$. For our purpose, we use the following results to allow us to use Theorem \ref{theorem:main_theorem} to study the PR of $S$-unit equations in 3 variables over $\Gamma$:

\begin{theorem}\cite[Theorem 1.1]{SchlickeweiSchmidt2000}, \cite[Corollary 6.4.5]{Evertse2015}\label{theorem:S-unit-equation}
    Let $G\subseteq \Com^{\times}\times \Com^{\times}$ be a finitely generated subgroup of rank $r$. Given $a,b\in \Com^{\times}$, the equation 
    \begin{equation*}
        ax+by=1
    \end{equation*}
    has at most $2^{8(r+2)}$ solutions $(x,y)\in G$.
\end{theorem}

If $\Gamma$ has rank $r$, then the rank of $\Gamma\times \Gamma$ is $2r$. Hence, an $S$-unit equation over $\Gamma$ has at most $2^{16(r+1)}$ solutions. This result yields the following:

\begin{corollary}\label{corollary:S-unit-equations_Non_PR}
Given any $a,b,c\in\Com^\times$, the equation $ax+by+cz=0$ is PR over $\Gamma$ if and only if it has constant solutions, namely if and only if $a+b+c=0$.
\end{corollary}

\begin{proof}
    Since $\Gamma$ is a multiplicative subgroup of $\Com^\times$, given any $s\in\Gamma$, solving $ax+by+cs=0$ in $\Gamma$ is equivalent to solving $a'x+b'y=1$, where $a'=-(a/cs)$ and $b'=-(b/cs)$. By Theorem \ref{theorem:S-unit-equation}, $a'x+b'y=1$ has a bounded number of solutions $k$ that only depends only on the rank of $\Gamma$. Hence, for all $s\in\Gamma$, the equation $ax+by+cs=0$ has at most $k$ solutions, so Theorem \ref{theorem:main_theorem} implies our thesis.
\end{proof}

Hence, the linear version of Rado's Theorem does not hold in any finite rank multiplicative subgroups of\footnote{By contrast, the group version of Rado's Theorem, with multiplication as operation, does hold for $\Gamma$, as it holds in any infinite group.} $\Gamma$ of $\Com^\times$. As a consequence, for example, $\Gamma$ will never contain arbitrarily long monochromatic arithmetic progressions; in fact, it cannot even contain monochromatic arithmetic progressions of length $3$, as these would form a solution of the linear equation $x+y-2z=0$.

\subsection{PR of Polyexponential Equations}

An algebraic number field is any finite field extension $K$ of $\Rat$, i.e. $K$ as a vector space over $\Rat$ has finite algebraic dimension; let $d$ be the degree of such extension, i.e. the algebraic dimension of $K$ over $\Rat$. As always, $K^\times$ denotes the multiplicative group $(K\setminus\{0\}, \cdot)$. Let $P_1,\dots,P_m\in K[x_1,\dots,x_n]$ be polynomials of degrees $d_1,\dots,d_m$, respectively.  Let $\boldsymbol{\alpha}_1,\dots,\boldsymbol{\alpha_m}\in (K^\times)^n$. 

The \emph{exponential polynomial} on the variables $\boldsymbol{x}=(x_1,\dots,x_n)$ over $\Ze$ with coefficients $P_1,\dots,P_m$ and characters $\boldsymbol{\alpha}_1,\dots,\boldsymbol{\alpha}_m$ is
\begin{equation*}\label{equation:exponential_polynomial_def}
    E(\boldsymbol{x}) = P_1(\boldsymbol{x})\boldsymbol{\alpha}_{1}^{\boldsymbol{x}}+\cdots P_m(\boldsymbol{x})\boldsymbol{\alpha}_{m}^{\boldsymbol{x}},\tag{$\star$}
\end{equation*}
where the \emph{exponential monomial} $\boldsymbol{\alpha}_{i}^{\boldsymbol{x}}$ is defined as
\begin{equation*}
    \boldsymbol{\alpha}_{i}^{\boldsymbol{x}}=\alpha_{i1}^{x_1}\cdots\alpha_{in}^{x_n}.
\end{equation*}
A solution of such an equation is any $\boldsymbol{a}\in\Ze^n$ such that $E(\boldsymbol{a})=0$. 

Henceforth we fix the characters $\boldsymbol{\alpha}_1,\dots,\boldsymbol{\alpha}_m\in\Ze^n$, polynomials in $n$ variables $P_1,\cdots,P_\in K[x_1,\dots,x_n]$, and functions $f_1,\dots,f_m:\Ze\to\Ze$; in this subsection we are interested in studying monochromatic solutions of the equation $E_{f_1,\dots,f_m}(\boldsymbol{x},y)=0$ over $\Ze$, where
\begin{equation*}
    E_{f_1,\dots,f_m}(\boldsymbol{x},y) = P_1(\boldsymbol{x})f_1(y)\alpha_{1}^{\boldsymbol{x}}+\cdots+P_m(\boldsymbol{x})f_m(y)\alpha_m^{\boldsymbol{x}}.
\end{equation*}
The basic theory about solutions of $E(\boldsymbol{x})=0$ can be found in \cite{SchlickeweiSchmidt2000}. We recall below some known facts from that article about such equations that will enable us to apply Theorem \ref{theorem:main_theorem} to study monochromatic solutions of $E_{f_1,\dots,f_k}(\boldsymbol{x},y)=0$.

Given a partition $\mathcal{P}$ of $[m]$ and $F\in\mathcal{P}$, let 
\begin{equation*}
    E_F(\boldsymbol{x})=\sum_{k\in F}P_k(\boldsymbol{x})\boldsymbol{\alpha}_{k}^{\boldsymbol{x}}.
\end{equation*}
We let $\sigma_\mathcal{P}(\boldsymbol{x})=\boldsymbol{0}$ to be the system consisting of all the equations of the form $E_F(\boldsymbol{x})=0$ as $F$ runs through $\mathcal{P}$. Clearly, if $\mathcal{Q}$ is a refinement\footnote{I.e. for all $F\in\mathcal{Q}$ one can find a $G\in\mathcal{P}$ satisfying $F\subseteq G$.} of the partition $\mathcal{P}$ then any solution of $\sigma_\mathcal{Q}(\boldsymbol{x})=\boldsymbol{0}$ is also a solution of $\sigma_\mathcal{P}(\boldsymbol{x})=\boldsymbol{0}$. In particular, a solution of $\sigma_\mathcal{P}(\boldsymbol{x})=\boldsymbol{0}$ is also a solution of the original equation $E(\boldsymbol{x})=0$; note also that a solution for the original equation has to be a solution for $\sigma_\mathcal{P}(\boldsymbol{x})=\boldsymbol{0}$ for some partition $\mathcal{P}$ of $[m]$. For a given $\mathcal{P}$, let $S(\mathcal{P})$ be the set of all solutions of $\sigma_\mathcal{P}(\boldsymbol{x})=\boldsymbol{0}$ that does not constitute a solution of $\sigma_\mathcal{Q}(\boldsymbol{x})=\boldsymbol{0}$ for any given proper refinement  $\mathcal{Q}$ of $\mathcal{P}$.

Given $\boldsymbol{\alpha}_{1},\dots,\boldsymbol{\alpha}_{m}\in\Ze^n$, $i,j\in[m]$ and a partition $\mathcal{P}$ of $[m]$, let $i\sim_\mathcal{P} j$ mean that there is a $F\in\mathcal{P}$ such that $i,j\in F$. Then the collection $G(\mathcal{P})$ of all $\boldsymbol{z}\in\Ze^n$ satisfying $\boldsymbol{\alpha}_i^{\boldsymbol{z}}=\boldsymbol{\alpha}_j^{\boldsymbol{z}}$ whenever $i\sim_\mathcal{P} j$, form an additive subgroup of $\Ze^n$. 

Define the constants 
\begin{equation*}
    A = \sum_{l=1}^{m}\binom{n+d_l}{n} \;\;\text{ and }\;\; B=\max(n,A).
\end{equation*}

\begin{theorem}\cite[Theorem 1]{SchlickeweiSchmidt2000}\label{theorem:number_sol_expo_pol_eq}
If $G(\mathcal{P})$ is trivial, then 
\begin{equation*}
    |S(\mathcal{P})| < 2^{35B^3}d^{6B^2}. 
\end{equation*}
\end{theorem}
In particular, if $G(\mathcal{P})$ is trivial for all partitions of $[m]$ then the initial equation admits a uniform bound on the number of solutions, namely  $B_m\cdot 2^{35B^3}d^{6B^2}$, where $B_m$ is the $m$-th Bell number (which is equal to the number of partitions of $[m]$), given by the recurrence relation $B_0=1$ and 
\begin{equation*}
    B_{k+1} = \sum_{l=0}^{k}\binom{k}{l}B_l.
\end{equation*}

Also observe that, if all the polynomials $P_1,\dots,P_m$ are constants, then $B=\max\{m,n\}$. For our purposes, the only fact that matters is that, when $G(\mathcal{P})$ is trivial for all partitions $\mathcal{P}$ of $[m]$, the amount of solutions is bounded by a constant that only depends on the number of variables, namely: the degree of $K$, the degree of the polynomials, and the number of monomials present in the equation. Hence, a direct consequence of Theorems \ref{theorem:main_theorem} and \ref{theorem:number_sol_expo_pol_eq} is:

\begin{theorem}\label{theorem:PR_exponential_equation}
    If the group $G(\mathcal{P})$ is trivial for all possible partitions of $[m]$, then the equation $E_{f_1,\dots,f_m}(\boldsymbol{x},y)=0$ is partition regular over $\Ze$ if and only if it admits constant solutions.  
\end{theorem}

An immediate example of such equations occurs when $\alpha_{11},\dots,\alpha_{mn}$ are mutually coprimes; in this case we have that, for each distinct $k,l\in[m]$ and $\boldsymbol{z}\in\Ze^n$, $\boldsymbol{\alpha}_{k}^{\boldsymbol{z}}=\boldsymbol{\alpha}_{l}^{\boldsymbol{z}}$ if and only if $\boldsymbol{z}=\boldsymbol{0}$.

\begin{corollary}\label{corollary:PR_exponential_coprimes}
Let $(\alpha_{lk})_{l=1,k=1}^{m,n}$ be a collection of mutually coprime integers, $P_1,\dots,P_m\in\Rat[x_1,\dots,x_n]$ and $f_1,\dots,f_m:\Ze\to\Ze$. For each $l\in[m]$ let $\boldsymbol{\alpha}=(\alpha_{l1},\dots,\alpha_{ln})$. For each $i\in[m]$, let $A_i(w)=P_i(w,\dots,w)f_i(w)$ and $a_i = \alpha_{i1}\cdots\alpha_{in}$. Then the polynomial exponential equation 
\begin{equation*}\label{equation:polyexpo_coprimes}
    P_1(\boldsymbol{x})f_1(y)\boldsymbol{\alpha}_{1}^{\boldsymbol{x}}+\dots+P_m(\boldsymbol{x})f_m(y)\boldsymbol{\alpha}_{m}^{\boldsymbol{x}} = 0
\end{equation*}
is partition regular over $\Ze$ if and only if there exists an $s\in\Ze$ such that 
\begin{equation*}
    a_1^s A_1(s) + \cdots + a_m^s A_m(s) = 0.
\end{equation*}
\end{corollary}

Another class of polyexponential equations we can study is the class of equations the form 
\begin{equation*}\label{equation:poly_exp_2}
    P_1(\boldsymbol{x},y)\boldsymbol{\alpha}_{1}^{\boldsymbol{x}}+\dots+P_m(\boldsymbol{x},y)\boldsymbol{\alpha}_{m}^{\boldsymbol{x}} = 0
\end{equation*}
where $P_1,\dots,P_m\in K[x_1,\dots,x_n,y]$. For a fixed list of such polynomials, characters, and an $s\in\Ze$, the equation 
\begin{equation*}\label{equation:poly_exp_2_1}
    P_1(\boldsymbol{x},s)\boldsymbol{\alpha}_{1}^{\boldsymbol{x}}+\dots+P_m(\boldsymbol{x},s)\boldsymbol{\alpha}_{m}^{\boldsymbol{x}} = 0\tag{$\star\star$}
\end{equation*}
is of the form of the Equation \eqref{equation:exponential_polynomial_def}. Thus, if the group $G(\mathcal{P})$ is trivial for all partitions of $[m]$, Theorem \ref{theorem:number_sol_expo_pol_eq} guarantees that the number of solutions of the Equation \eqref{equation:poly_exp_2_1} is bounded uniformly, i.e. independently from the values of $s$. Hence, by Theorem \ref{theorem:main_theorem}, we prove the following:

\begin{theorem}\label{theorem:PR_exponential_equation_}
    If the group $G(\mathcal{P})$ is trivial for all possible partitions of $[m]$, then the equation 
\begin{equation*}\label{equation:poly_exp_3}
    P_1(\boldsymbol{x},y)\boldsymbol{\alpha}_{1}^{\boldsymbol{x}}+\dots+P_m(\boldsymbol{x},y)\boldsymbol{\alpha}_{m}^{\boldsymbol{x}} = 0
\end{equation*} 
is partition regular over $\Ze$ if and only if it admits constant solutions.  
\end{theorem}
As before, a basic consequence of the above Theorem is when the characters are coprime:
\begin{corollary}\label{corollary:polyexp_2}
    Let $(\alpha_{lk})_{l=1,k=1}^{m,n}$ be a collection of mutually coprime integers and $P_1,\dots,P_m\in\Rat[x_1,\dots,x_n,y]$. For each $l\in[m]$ let $\boldsymbol{\alpha}=(\alpha_{l1},\dots,\alpha_{ln})$. For each $i\in[m]$, let $\tilde{P}_i(w)=P_i(w,\dots,w)$ and $a_i = \alpha_{i1}\cdots\alpha_{in}$. Then the polynomial exponential equation 
\begin{equation*}
    P_1(\boldsymbol{x},y)\boldsymbol{\alpha}_{1}^{\boldsymbol{x}}+\dots+P_m(\boldsymbol{x},y)\boldsymbol{\alpha}_{m}^{\boldsymbol{x}} = 0
\end{equation*}
is partition regular over $\Ze$ if and only if there exists $s\in\Ze$ such that 
\begin{equation*}
    a_1^s \tilde{P}_1(s) + \cdots + a_m^s \tilde{P}_m(s) = 0.
\end{equation*}
\end{corollary}

\begin{example}
    Let $P_1(x,y,z) = xy-z+2$,  $P_2(x,y,z)=x-y+2z+2$,  and  $P_3(x,y,z)=xy-z+3$. Then the polyexponential equation 
\begin{equation*}
    P_1(x,y,z)2^x3^y + P_2(x,y,z)5^x7^y+P_3(x,y,z)11^x13^y = 0
\end{equation*}
is not PR over $\Ze$. Indeed, if the contrary happens, there exists $s\in\Ze$ such that
\begin{equation*}\label{equation:example_polyexp}
    6^s\Tilde{P}_1(s) + 35^s\Tilde{P}_2(s)+143^s\Tilde{P}_3(s)=0.\tag{$\dagger$}
\end{equation*}
Since
\begin{equation*}
    \Tilde{P}_1(w)=w^2-w+2, \;\; \Tilde{P}_2(w)=2w+2\;\;\text{ and }\;\;\Tilde{P}_3(w)=w^2-w+3,
\end{equation*}
we have that $2$ divides $\Tilde{P}_1(n)$ and $\Tilde{P}_2(n)$ for all $n\in\Ze$. Nevertheless, if there is a common prime factor $p$ of $\Tilde{P}_1(s),\Tilde{P}_2(s)$ and $\Tilde{P}_3(s)$, then we would have 
\begin{equation*}
    x^2-x+2\equiv x^2-x+3 \mod p,
\end{equation*}
implying that $2\equiv 3\mod p$, which is absurd. Hence, 
\begin{equation*}
    d=\gcd(\Tilde{P}_1(s),\Tilde{P}_2(s))>1
\end{equation*}
does not divide $\Tilde{P}_3(s)$. By B\'{e}zoult's Lemma, this a contradiction with the equation \eqref{equation:example_polyexp}.
\end{example}

\subsection*{Acknowledgements}
L. Luperi Baglini was supported by the project PRIN 2022 "Logical methods in combinatorics", 2022BXH4R5, MIUR (Italian Ministry of University and Research). The authors thank the anonymous reviewers for their helpful comments and suggestions.

\end{document}